\documentclass{amsart}

\usepackage{amsfonts}
\usepackage{amsmath}
\usepackage{amssymb}
\usepackage{setspace}
\usepackage{graphicx}
\usepackage{amscd,amssymb,amsthm}

\usepackage{hyperref}
\usepackage{graphicx}
\usepackage{amsmath, amsthm, amscd, amsfonts, amssymb, graphicx, color}
 \usepackage{url}
\usepackage{enumerate}
 \makeatletter
\let\reftagform@=\tagform@
\def\tagform@#1{\maketag@@@{(\ignorespaces\textcolor{blue}{#1}\unskip\@@italiccorr)}}
\renewcommand{\eqref}[1]{\textup{\reftagform@{\ref{#1}}}}
\makeatother
\usepackage{hyperref}
\hypersetup{colorlinks=true, linkcolor=red, anchorcolor=green,
citecolor=cyan, urlcolor=red, filecolor=magenta, pdftoolbar=true}

\usepackage{epsfig}        
\usepackage{epic,eepic}       
\newtheorem{theorem}{Theorem}
\theoremstyle{plain}

\newtheorem{lemma}{Lemma}

\newtheorem{remark}{Remark}

\numberwithin{equation}{section}

\DeclareMathOperator{\erf}{erf}
 
  \DeclareMathOperator{\lip}{Lip}

\begin{document}

\title[$L^p$--ERROR BOUNDS OF TWO AND
THREE--POINT  QUADRATURE]{$L^p$--ERROR BOUNDS OF TWO AND
THREE--POINT QUADRATURE RULES FOR RIEMANN--STIELTJES INTEGRALS}

\author[M.W. Alomari]{Mohammad W. Alomari$^1$}

\address{$^1$Department of Mathematics, Faculty of Science and
Information Technology, Irbid National University, P.O. Box 2600,
P.C. 21110, Irbid, Jordan.} \email{mwomath@gmail.com}

\author[A. Guessab]{Allal  Guessab$^2$}

\address{$^2$Laboratoire de Math\'{e}matiques et de leurs Applications, UMR CNRS
4152, Universit\'{e} de Pau et des Pays de l'Adour, 64000 Pau,
France} \email{allal.guessab@univ-pau.fr}

\date{\today}


\begin{abstract}
In this work,  $L^p$-error estimates of general two and three
point quadrature rules for  Riemann-Stieltjes integrals are given.
The presented proofs depend  on new triangle type inequalities of
Riemann-Stieltjes integrals.
\end{abstract}

\maketitle

\section{Introduction}

The Newton--Cotes formulas use values of function at equally
spaced points. The same practice when the formulas are combined to
form the composite rules, but this restriction can significantly
decrease the accuracy of the approximation. In fact, these methods
are inappropriate when integrating a function on an interval that
contains both regions with large functional variation and regions
with small functional variation. If the approximation error is to
be evenly distributed, a smaller step size is needed for the
large-variation regions than for those with less variation.

In numerical analysis, inequalities play a main role in error
estimations. A few years ago, by using modern theory of
inequalities and Peano kernel approach a number of authors have
considered an error analysis of some quadrature rules of
Newton-Cotes type. In particular, the Mid-point, Trapezoid,
Simpson's and other rules have been investigated  recently with
the view of obtaining bounds for the quadrature rules in terms of
at most first derivative.

The number of proposed quadrature rules that provides
approximation of Stieltjes integral  $\int_a^b {f\left( t
\right)du\left( t \right)}$  using derivatives or without using
derivatives are very rare in comparison with the large number of
methods available to approximate the classical Riemann integral $
\int_{a}^{b}{f\left( t\right) dt}$.

The problem of introducing quadrature rules for
$\mathcal{RS}$-integral $\int_a^b {fdg}$ was studied via theory of
inequalities by many authors. Two famous real inequalities were
used in this approach, which are the well known Ostrowski  and
Hermite-Hadamard inequalities and their modifications. For this
purpose and in order to approximate the $\mathcal{RS}$-integral
$\int_a^b {f\left( t \right)du\left( t \right)}$,   a
generalization of closed Newton-Cotes quadrature rules of
$\mathcal{RS}$-integrals without using derivatives provides a
simple and robust solution to a significant problem in the
evaluation of certain applied probability models was presented by
Tortorella in \cite{Tortorella}.

In 2000, Dragomir \cite{Dragomir2} introduced the Ostrowski's
approximation formula (which is of One-point type formula) as
follows:
\begin{align*}
\int_a^b {f\left( t \right)du\left( t \right)}\cong f\left( x
\right) \left[ {u\left( {b} \right) - u\left( a \right)} \right]
\qquad \forall x\in [a,b].
\end{align*}
Several error estimations for this approximation had been done in
the works \cite{Dragomir1} and \cite{Dragomir2}.

From different point of view, the authors of \cite{Dragomir5} (see
also \cite{Barnett,Barnett1})  considered the problem of
approximating the Stieltjes integral $\int_a^b {f\left( t
\right)du\left( t \right)}$ via the generalized trapezoid formula:
\begin{align*}
\int_a^b {f\left( t \right)du\left( t \right)}\cong \left[
{u\left( x \right) - u\left( a \right)} \right]f\left( a \right) +
\left[ {u\left( b \right) - u\left( x \right)} \right]f\left( b
\right).
\end{align*}
Many authors have studied this quadrature rule under various
assumptions of integrands and integrators. For full history of
these two quadratures see \cite{alomari3} and the references
therein.

Another trapezoid type formula was considered in \cite{Dragomir8},
which reads:
\begin{align*}
\int_a^b {f\left( t \right)du\left( t \right)}\cong \frac{f\left(
a \right) + f\left( b \right)}{2}\left[ {u\left( {b} \right) -
u\left( a \right)} \right] \qquad \forall x\in [a,b].
\end{align*}
Some related results had been presented by the same author in
\cite{Dragomir6} and \cite{Dragomir7}. For other connected results
see \cite{CeroneDragomir} and \cite{CeroneDragomir1}.

In 2008, Mercer \cite{Mercer}  introduced the following trapezoid
type formula for the $\mathcal{RS}$-integral
\begin{align}
\label{Mercer.Q}\int_a^b {fdg} \cong \left[ {G - g\left( a
\right)} \right]f\left( a \right) + \left[ {g\left( b \right) - G}
\right]f\left( b \right),
\end{align}
 where $G=
\frac{1}{{b  -a}}\int_a^b  {g\left( t \right)dt}$.

Recently, Alomari and Dragomir \cite{alomari1}, proved several new
error bounds for the Mercer--Trapezoid quadrature rule
(\ref{Mercer.Q}) for the $\mathcal{RS}$-integral under various
assumptions involved the integrand $f$ and the integrator $g$.
\newline

Follows Mercer approach in \cite{Mercer}, Alomari and Dragomir
\cite{alomari6} introduced the following three-point quadrature
formula:
\begin{align}
\label{error.term} \int_a^b {f\left( t \right)dg \left( t \right)}
&\cong \left[ {G\left( {a,x} \right) - g\left( a \right)}
\right]f\left( a \right) + \left[ {G\left( {x,b} \right) - G\left(
{a,x} \right)} \right]f\left( x \right)
\nonumber\\
&+ \left[ {g\left( b \right) - G\left( {x,b} \right)}
\right]f\left( b \right)
\end{align}
for all $a<x<b$, where $G\left( {\alpha ,\beta } \right): =
\frac{1}{{\beta  - \alpha }}\int_\alpha ^\beta  {g\left( t
\right)dt}$.

\noindent Several error estimations of Mercer's type quadrature
rules for $\mathcal{RS}$-integral under various assumptions about
the function involved have been considered in \cite{alomari1} and
\cite{alomari4}.

Motivated by Guessab-Schmeisser inequality (see \cite{Guessab})
which is of Ostrowski's type, Alomari in \cite{alomari2} and
\cite{alomari5} presented the following approximation formula for
$\mathcal{RS}$-integrals:
\begin{multline}
\int_a^b {f\left( t \right)du\left( t \right)}
\\
\cong\left[ {u\left( {\frac{{a + b}}{2}} \right) - u\left( a
\right)} \right]f\left( x \right) + \left[ {u\left( b \right) -
u\left( {\frac{{a + b}}{2}} \right)} \right]f\left( {a + b - x}
\right),
\end{multline}
for all $x \in \left[ {a,\frac{{a + b}}{2}}\right]$. For other
related results see \cite{alomari3}. For different approaches
variant quadrature formulae the reader may refer to \cite{M},
\cite{Gautschi} and  \cite{Munteanu}.

Among others the $L^{\infty}$-norm gives the highest possible
degree of precision; so that it is recommended to be `almost' the
norm of choice.  However, in some cases we cannot access the
$L^{\infty}$-norm, so that  $L^p$-norm  ($1\le p < \infty$) is
considered to be a variant norm  in error estimations.

In this work, several $L^p$-error estimates ($1\le p < \infty$) of
general Two and Three points quadrature rules for
Riemann-Stieltjes integrals are presented. The presented proofs
depend on new triangle type inequalities for
$\mathcal{RS}$-integrals.

\section{Two Lemmas}\label{sec2}

It is well known that the class of functions satisfying Lipschitz
condition  is a subset of the class of functions of bounded
variation. More preciously, if $f$ has the Lipschitz property,
then $f$ is of bounded variation. However, a continuous function
of bounded variation  need not have a Lipschitz property. For
example,  the series $ \sum_{k=1}^{\infty} {\frac{\sin kt}{k \log
k}}$ $(0 \le t \le 1)$ converges uniformly to the sum $g$, which
is absolutely continuous and hence is of bounded variation,
however  $g$ does not satisfies Lipschitz property.

Not far away from this,  very useful inequality regarding
Lipschitz functions is the following: for a Riemann integrable
function $w:[a,b] \to \mathbb{R}$ and $L$--Lipschitzian function
$\nu:[a,b] \to \mathbb{R}$, one has the inequality
\begin{align}
\left| {\int_a^b {w\left( t \right)d\nu\left( t \right)} } \right|
\le L\left\| w \right\|_1. \label{eq2.1}
\end{align}
A generalization of this inequality to $L^p$-spaces is
incorporated in the following lemma \cite{MA}:
\begin{lemma}\label{lemma1}
Let $1 \le p < \infty$. Let $w,\nu : [a,b] \to \mathbb{R}$ be such
that is $w \in L^p[a,b]$ and $\nu$ has a Lipschitz property on
$[a,b]$. Then the inequality
\begin{align}
\label{eq2.2}\left| {\int_a^b {w\left( t \right)d\nu\left( t
\right)} } \right| \le L \left( {b-a} \right)^{1 - {\textstyle{1
\over p}}}\cdot  \left\| w \right\|_p,
\end{align}
holds and the constant $`1$' in the right hand side is the best
possible. Provided that the $\mathcal{RS}$-integral $ \int_a^b
{w\left( t \right)d\nu\left( t \right)}$ exists, where
\begin{align*}
\left\| w \right\|_p  = \left( {\int_a^b {\left| {w\left( t
\right)} \right|^p dt} } \right)^{1/p}, \qquad (1\le p \le
\infty).
\end{align*}
\end{lemma}

\begin{remark}
Clearly, when $p=1$ in \eqref{eq2.2} then we refer to
\eqref{eq2.1}.
\end{remark}

Under weaker conditions we may state the following result \cite{MA}:
\begin{lemma}\label{lemma2}
\emph{Let $1 \le p < \infty$. Let $w,\nu : [a,b] \to \mathbb{R}$
be such that is $w \in L^p[a,b]$ and $\nu$ is of bounded variation
on $[a,b]$. Then the  inequality
\begin{align}
\label{eq2.3}\left| {\int_a^b {w\left( t \right)d\nu\left( t
\right)} } \right| \le \left( {\bigvee_a^b\left( \nu \right)}
\right)^{1 - {\textstyle{1 \over p}}} \cdot \left|{\nu^{\prime}
}\right\|_{\infty} \cdot \left\| w \right\|_p, \qquad a.e.
\end{align}
holds. The constant $`1$' in the right hand side is the best
possible. Provided that the $\mathcal{RS}$-integral $ \int_a^b
{w\left( t \right)d\nu\left( t \right)}$ exists.}
\end{lemma}

 \begin{remark}
\emph{If $\nu$ is $M$-Lipschitz then}
\begin{align*}
\lip_{\rm{M}}\left(\nu\right)=\mathop {\sup }\limits_{x,y \in
\left[ {a,b} \right]} \left| {\frac{{\nu\left( y \right) -
\nu\left( x \right)}}{{y - x}}} \right| < \infty.
\end{align*}
\emph{Therefore, we rewrite the inequality \eqref{eq2.3} such as:}
\begin{align}
\label{eq2.4}\left| {\int_a^b {w\left( t \right)d\nu\left( t
\right)} } \right| \le \lip_{\rm{M}} \left( \nu \right) \left(
{\bigvee_a^b\left( \nu \right)} \right)^{1 - {\textstyle{1 \over
p}}}  \cdot \left\| w \right\|_p, \qquad
\end{align}
\emph{which is valid everywhere and sharp.}
 \end{remark}

\section{ A General Quadrature rule For
$\mathcal{RS}$-integrals}\label{sec3}

Let $\Phi _\alpha  \left( {f,u;x} \right)$ is the general
quadrature formula
\begin{multline}
\Phi _\alpha  \left( {f,u;x} \right)
 := \left( {1 - \alpha } \right)\left\{ {\left[ {u\left( {\frac{{a +
b}}{2}} \right) - u\left( a \right)} \right]f\left( x \right)+
\left[ {u\left( b \right) - u\left( {\frac{{a + b}}{2}} \right)}
\right]f\left( {a + b - x} \right)} \right\}
\\
+\alpha \left[ {\left( {u\left( x \right) - u\left( a \right)}
\right)f\left( a \right) + \left( {u\left( b \right) - u\left( x
\right)} \right)f\left( b \right)} \right].\label{eq3.1}
\end{multline}
Define the mapping
\begin{align*}
S_{u} \left( {t;x} \right): = \left\{ \begin{array}{l} \left( {1 -
\alpha } \right)\left[ {u\left( t \right) - u\left( a \right)}
\right] + \alpha \left[ {u\left( t \right) - u\left( x
\right)} \right],\,\,\,\,\,\,\,\,\,\,\,\,\,\,t \in \left[ {a,x} \right] \\
  \\
\left( {1 - \alpha } \right)\left[ {u\left( t \right) - u\left(
{\frac{{a + b}}{2}} \right)} \right] + \alpha \left[ {u\left( t
\right) - u\left( x \right)} \right],\,\,\,\,\,\,\,t \in \left( {x,a + b - x} \right] \\
  \\
\left( {1 - \alpha } \right)\left[ {u\left( t \right) - u\left( b
\right)} \right] + \alpha \left[ {u\left( t \right) - u\left( x
\right)} \right],\,\,\,\,\,\,\,\,\,\,\,\,\,\,t \in \left( {a + b - x,b} \right] \\
 \end{array} \right..
\end{align*}
Using integration by parts formula, its not difficult to obtain
that
\begin{align}
\int_a^b {S_{u} \left( {t;x} \right)df\left( t
\right)}=\Phi_\alpha  \left( {f,u;x} \right)- \int_a^b {f\left(
{t} \right)du\left( t \right)} = \mathcal{E}_\alpha  \left(
{f,u;x} \right).\label{eq3.2}
\end{align}
where  $\mathcal{E}_\alpha  \left( {f,u;x} \right)$ is the error
term.

Thus, the $\mathcal{RS}$-integral $ \int_a^b {f\left( t
\right)du\left( t \right)}$ can be approximated by the quadrature
rule
\begin{align}
\int_a^b {f\left( t \right)du\left( t \right)} = \Phi _\alpha
\left( {f,u;x} \right) - \mathcal{E}_\alpha  \left( {f,u;x}
\right).\label{eq3.3}
\end{align}
In particular cases, we consider:
\begin{itemize}
\item If $\alpha = 0$, then the following general Two-point
formula holds
\begin{multline}
\Phi_0  \left( {f,u;x} \right) =  \left[ {u\left( {\frac{{a +
b}}{2}} \right) - u\left( a \right)} \right]f\left( x \right) +
\left[ {u\left( b \right) - u\left( {\frac{{a + b}}{2}} \right)}
\right]f\left( {a + b - x} \right).
\end{multline}

\item If $\alpha = \frac{1}{3}$, then the following general
Three-point formula holds
\begin{align}
\Phi_{\frac{1}{3}}  \left( {f,u;x} \right) &=\frac{1}{3} \left[
{\left( {u\left( x \right) - u\left( a \right)} \right)f\left( a
\right) + \left( {u\left( b \right) - u\left( x \right)}
\right)f\left( b \right)} \right]
\\
&\qquad+ \frac{2}{3}\left\{ {\left[ {u\left( {\frac{{a + b}}{2}}
\right) - u\left( a \right)} \right]f\left( x \right)+\left[
{u\left( b \right) - u\left( {\frac{{a + b}}{2}} \right)}
\right]f\left( {a + b - x} \right)} \right\}\nonumber
\end{align}

\item If $\alpha = \frac{1}{2}$, then the following general
Average Trapezoid-Midpoint formula holds
\begin{align}
\Phi_{\frac{1}{2}}  \left( {f,u;x} \right) &= \frac{1}{2}\left\{
{\left( {u\left( x \right) - u\left( a \right)} \right)f\left( a
\right) + \left( {u\left( b \right) - u\left( x \right)}
\right)f\left( b \right)}\right.
\\
&\qquad\left.{+ \left[ {u\left( {\frac{{a + b}}{2}} \right) -
u\left( a \right)} \right]f\left( x \right)+\left[ {u\left( b
\right) - u\left( {\frac{{a + b}}{2}} \right)} \right]f\left( {a +
b - x} \right)} \right\}.\nonumber
\end{align}

\item If $\alpha = 1$, then the following general Trapezoid
formula holds
\begin{align}
\Phi_{1}  \left( {f,u;x} \right) := \left[ {u\left( x \right) -
u\left( a \right)} \right]f\left( a \right) + \left[ {u\left( b
\right) - u\left( x \right)} \right]f\left( b \right).
\end{align}
\end{itemize}

A convex combination between Trapezoid and Midpoint formulas is
incorporated in the relation:
\begin{align*}
\Phi _\alpha  \left( {f,u;\frac{a+b}{2}} \right)
 &= \alpha \left\{ {\left[ {u\left( \frac{a+b}{2} \right) - u\left( a
\right)} \right]f\left( a \right) + \left[ {u\left( b \right) -
u\left( \frac{a+b}{2} \right)} \right]f\left( b \right)} \right\}
\\
&\qquad+\left( {1 - \alpha } \right) \left[ {u\left( b \right) -
u\left( {a} \right)} \right]f\left( {\frac{a+b}{2}} \right),
\end{align*}
for all $\alpha \in \left[0,1\right]$. Furthermore, if $\alpha
=\frac{1}{3}$, the we get the Simpson's formula for
$\mathcal{RS}$-integrals.

\begin{theorem}\label{thm1}
Let $u: [a,b] \to [0, \infty)$ be  a H\"{o}lder continuous of
order $r \in (0,1]$  on $[a,b]$ and belongs to $L_p [a,b]$
$(p\ge1)$. If $f:[a,b] \to \mathbb{R}$ $M$--Lipschitzian mapping
on $[a,b]$, then for any $x \in \left[a,\frac{a+b}{2}\right]$ and
$\alpha \in [0,1]$, we have
\begin{align}
\label{mr1}&\left| {\mathcal{E}_\alpha  \left( {f,u;x} \right)}
\right|
\\
&\le H \lip_{\rm{M}} \left( f \right) \cdot\left(
{\bigvee_a^b\left( f \right)} \right)^{1 - {\textstyle{1 \over
p}}} \cdot \left[{\left(1-\alpha\right) \left\{{
2\frac{\left(x-a\right)^{\frac{rp+1}{p}}}{\left({rp+1}\right)^{1/p}}+2^{1/p}\frac{
\left(\frac{a+b}{2}-x\right)^{\frac{rp+1}{p}}}{\left({rp+1}\right)^{1/p}}}\right\}
}\right.
\nonumber\\
&\qquad+ \left.{\alpha
 \left\{{\frac{\left(x-a\right)^{\frac{rp+1}{p}}}{\left({rp+1}\right)^{1/p}}+
\frac{\left(a+b-2x\right)^{\frac{rp+1}{p}}}{\left({rp+1}\right)^{1/p}}
+\left({\frac{\left(b-x\right)^{rp+1}-\left(a+b-2x\right)^{rp+1}}{rp+1}}\right)^{1/p}}\right\}
}\right]\nonumber
\end{align}
\end{theorem}

\begin{proof}
As $f$ is of bounded variation on $[a,b]$, and $u$ is H\"{o}lder
continuous of order $r \in (0,1]$ which belongs to $L^p[a,b]$,
then by \eqref{eq2.4} we have
\begin{align*}
&\left| {\int_a^b {S_{u} \left( {t;x} \right)df\left( t \right)}}
\right|
\\
&=\left| {\int_a^x { \left\{ {\left( {1 - \alpha } \right)\left[
{u\left( t \right) - u\left( a \right)} \right] + \alpha \left[
{u\left( t \right) - u\left( x \right)} \right]} \right\}df\left(
t \right)}} \right|
\nonumber\\
&\qquad+\left| {\int_x^{a+b-x} {\left\{ {\left( {1 - \alpha }
\right)\left[ {u\left( t \right) - u\left( {\frac{{a + b}}{2}}
\right)} \right] + \alpha \left[ {u\left( t \right) - u\left( x
\right)} \right]} \right\}df\left( t \right)}} \right|
\nonumber\\
&\qquad+\left| {\int_{a+b-x}^b {\left\{ {\left( {1 - \alpha }
\right)\left[ {u\left( t \right) - u\left( b \right)} \right] +
\alpha \left[ {u\left( t \right) - u\left( x \right)} \right]}
\right\}df\left( t \right)}} \right|
\\
&\le  \lip_{\rm{M}}\left( f \right)  \left\{{  \left(
{\bigvee_a^x\left( f \right)} \right)^{1 - {\textstyle{1 \over
p}}}  \times \left[{\left( {1 - \alpha } \right) \left\| u -
u\left( a \right) \right\|_{p,[a,x]} + \alpha \left\| u - u\left(
x \right) \right\|_{p,[a,x]} }\right]}\right.
\nonumber\\
&\qquad+\left( {\bigvee_x^{a+b-x}\left( f \right)} \right)^{1 -
{\textstyle{1 \over p}}}
 \times \left[{\left( {1 - \alpha } \right) \left\| u
- u\left( {\frac{a+b}{2}}\right) \right\|_{p,[x,a+b-x]} + \alpha
\left\| u - u\left( x \right) \right\|_{p,[x,a+b-x]} }\right]
\\
&\qquad\left.{+ \left( {\bigvee_{a+b-x}^b\left( f \right)}
\right)^{1 - {\textstyle{1 \over p}}}
 \times \left[{\left( {1 - \alpha } \right) \left\| u
- u\left( b \right) \right\|_{p,[a+b-x,b]} + \alpha \left\| u -
u\left( x \right) \right\|_{p,[a+b-x,b]} }\right] }\right\}
\\
&\le \lip_{\rm{M}}\left( f \right) \left( {\bigvee_a^b\left( f
\right)} \right)^{1 - {\textstyle{1 \over p}}}
\nonumber\\
&\qquad\times \left( {1 - \alpha } \right)\left[{ \left\| u -
u\left( a \right) \right\|_{p,[a,x]} +\left\| u - u\left(
{\frac{a+b}{2}}\right) \right\|_{p,[x,a+b-x]}+ \left\| u - u\left(
b \right) \right\|_{p,[a+b-x,b]} }\right]
\nonumber\\
&\qquad\qquad+ \alpha \left\{{ \left\| u - u\left( x \right)
\right\|_{p,[a,x]}+\left\| u - u\left( x \right)
\right\|_{p,[x,a+b-x]}+\left\| u - u\left( x \right)
\right\|_{p,[a+b-x,b]} } \right\}.
\end{align*}
Now, since $u$ is H\"{o}lder continuous of order $r\in (0,1]$,
then there exits a positive constant $H>0$ such that
\begin{align*}
 \left|{u\left(y\right)-u\left(z\right)}\right|
\le H\left|{y-z}\right|^r
\end{align*}
for all $y,z\in [a,b]$. Accordingly, since $u\in L^p[a,b]$ then
for all fixed $z\in[c,d]$ we have
\begin{align*}
\left\|{u-u\left(z\right)}\right\|^p_{p,[c,d]}=
 \int_c^d{\left|{u\left(y\right)-u\left(z\right)}\right|^pdy} &\le
H^p\int_c^d{\left|{y-z}\right|^{rp}dy}
\\
&= H^p
\frac{\left(z-c\right)^{rp+1}+\left(d-z\right)^{rp+1}}{rp+1}
\end{align*}
for every subinterval $\left[{c,d}\right] \subseteq
\left[{a,b}\right]$ and $y,z \in \left[{c,d}\right]$. Applying
this step for each norm in the last inequality above, we get
\begin{align*}
&\left| {\int_a^b {S_{u} \left( {t;x} \right)df\left( t \right)}}
\right|
\\
&\le H\lip_{\rm{M}}\left( f \right) \cdot \left(
{\bigvee_a^b\left( f \right)} \right)^{1 - {\textstyle{1 \over
p}}} \cdot   \left[{\left(1-\alpha\right) \left\{{
2\frac{\left(x-a\right)^{\frac{rp+1}{p}}}{\left({rp+1}\right)^{1/p}}+2^{1/p}\frac{
\left(\frac{a+b}{2}-x\right)^{\frac{rp+1}{p}}}{\left({rp+1}\right)^{1/p}}}\right\}
}\right.
\\
&\qquad+
\left.{\alpha\left\{{\frac{\left(x-a\right)^{\frac{rp+1}{p}}}{\left({rp+1}\right)^{1/p}}+
\frac{\left(a+b-2x\right)^{\frac{rp+1}{p}}}{\left({rp+1}\right)^{1/p}}
+\left({\frac{\left(b-x\right)^{rp+1}-\left(a+b-2x\right)^{rp+1}}{rp+1}}\right)^{1/p}}\right\}
}\right],
\end{align*}
and hence the proof is established.
\end{proof}

\begin{remark}
In Theorem \ref{thm1}, if $u \in L^2[a,b]$ then  for all $x \in
\left[{a,\frac{a+b}{2} }\right]$ and $\alpha \in [0,1]$, we have
\begin{align}
\label{mr3}&\left| {\mathcal{E}_\alpha  \left( {f,u;x} \right)}
\right|
\\
&\le H \lip_{\rm{M}}\left( f \right) \cdot \left(
{\bigvee_a^b\left( f \right)} \right)^{\frac{1}{2}} \cdot \left[{
\left(1-\alpha\right)\left\{{
2\frac{\left(x-a\right)^{\frac{2r+1}{2}}}{\left({2r+1}\right)^{1/2}}+2^{1/2}\frac{
\left(\frac{a+b}{2}-x\right)^{\frac{2r+1}{2}}}{\left({2r+1}\right)^{1/2}}}\right\}
}\right.
\nonumber\\
&\qquad+
\left.{\alpha\left\{{\frac{\left(x-a\right)^{\frac{2r+1}{2}}}{\left({2r+1}\right)^{1/2}}+
\frac{\left(a+b-2x\right)^{\frac{2r+1}{2}}}{\left({2r+1}\right)^{1/2}}
+\left({\frac{\left(b-x\right)^{2r+1}-\left(a+b-2x\right)^{2r+1}}{2r+1}}\right)^{1/2}}\right\}
}\right].\nonumber
\end{align}
Moreover, if $u$ is Lipschitzian mapping (i.e., $r=1$), we get
\begin{align}
\label{mr4}&\left| {\mathcal{E}_\alpha  \left( {f,u;x} \right)}
\right|
\\
&\le \frac{ 1}{\sqrt{3}} H \lip_{\rm{M}}\left( f \right) \cdot
\left( {\bigvee_a^b\left( f \right)} \right)^{\frac{1}{2}} \cdot
\left[{\left(1-\alpha\right) \left\{{
2\left(x-a\right)^{\frac{3}{2}}
+2^{1/2}\left(\frac{a+b}{2}-x\right)^{\frac{3}{2}} }\right\}
}\right.
\nonumber\\
&\qquad+ \left.{\alpha\left\{{\left(x-a\right)^{\frac{3}{2}}+
\left(a+b-2x\right)^{\frac{3}{2}}
+\left({\left(b-x\right)^{3}-\left(a+b-2x\right)^{3}}\right)^{1/2}}\right\}
}\right]\nonumber
\end{align}
\end{remark}

\begin{remark}
In very special interesting case if $u$ is H\"{o}lder continuous
of order $r=\frac{1}{p}$ ($p \ge 1$) and belongs to $L^p[a,b]$,
then \eqref{mr1} becomes
\begin{align}
\label{mr5}&\left| {\mathcal{E}_\alpha  \left( {f,u;x} \right)}
\right|
\\
&\le \frac{H}{2^{1/p}} \lip_{\rm{M}}\left( f \right)\cdot \left(
{\bigvee_a^b\left( f \right)} \right)^{1 - {\textstyle{1 \over
p}}} \cdot \left[{\left(1-\alpha\right) \left\{{
2\left(x-a\right)^{\frac{2}{p}} +2^{1/p}
\left(\frac{a+b}{2}-x\right)^{\frac{2}{p}} }\right\} }\right.
\nonumber\\
&\qquad+ \left.{\alpha
 \left\{{\left(x-a\right)^{\frac{2}{p}}+ \left(a+b-2x\right)^{\frac{2}{p}}
+\left({\left(b-x\right)^{2}-\left(a+b-2x\right)^{2}}\right)^{1/p}}\right\}
}\right]\nonumber
\end{align}
\end{remark}

 \begin{remark}
In \eqref{mr1}--\eqref{mr4}, choosing an appropriate $x\in
\left[{a,\frac{a+b}{2}}\right]$ we get error estimations of
several quadrature formulae for $\mathcal{RS}$-integrals, such as:
Trapezoid, several Two-points, Midpoint, Simpson's, Three-point,
Average Trapezoid-Midpoint quadrature formulae and others. In
parallel, these inequalities may be considered as generalizations
of Ostrowski's type inequalities for $\mathcal{RS}$-integrals for
arbitrary $x\in \left[{a,\frac{a+b}{2}}\right]$.
\end{remark}

\section{More Error bounds in $L^p$-space}\label{sec4}

Let $I$ be a real interval such that $[a,b] \subseteq I^{\circ}$
the interior of $I$,$a,b\in \mathbb{R}$ $a<b$. Consider
$\mathfrak{U}^p(I)$ ($p>1$) be the space of all positive $n$-th
differentiable functions $f$ whose $n$-th derivatives $f^{(n)}$ is
positive locally absolutely continuous on $I^{\circ}$ with $
\int_a^b {\left( {f^{(n)}\left( t \right)} \right)^p dt} < \infty
$.

\begin{theorem}\label{thm2}
Let $u\in \mathfrak{U}^p(I)$. Assume that $f:[a,b] \to \mathbb{R}$
is $M$--Lipschitz on $[a,b]$, then for any $x \in
\left[a,\frac{a+b}{2}\right]$ and $\alpha \in [0,1]$, we have
\begin{align}
&\left| {\mathcal{E}_\alpha  \left( {f,u;x} \right)} \right|
\\
&\le \lip_{\rm{M}}\left( f \right)\cdot \left( {\bigvee_a^b\left(
f \right)} \right)^{ 1-{\textstyle{1 \over p}}}    \cdot   \left(
{\frac{{p\sin \left( {{\textstyle{\pi \over p}}} \right)}}{{\pi
\sqrt[p]{{p - 1}}}}} \right)^n
\nonumber\\
&\qquad\times \left[ {\left( {1 - \alpha } \right)\left( {\frac{{b
- a}}{4} + \left| {x - \frac{{3a + b}}{4}} \right|} \right)^n +
\alpha  \cdot \left( {\frac{{b - a}}{2} + \left| {x - \frac{{a +
b}}{2}} \right|} \right)^n} \right] \cdot \left\| {u^{(n)}}
\right\|_{p,\left[ {a,b} \right]},\nonumber
\end{align}
for all $x \in \left[ {a,\frac{{a + b}}{2}}\right]$. In
particular, we have
\begin{multline}
\left| {\mathcal{E}_\alpha  \left( {f,u; \frac{a+b}{2}} \right)}
\right|
\\
\le \lip_{\rm{M}}\left( f \right)\cdot\left( {\frac{{p\sin \left(
{{\textstyle{\pi  \over p}}} \right)}}{{\pi \sqrt[p]{{p - 1}}}}}
\right)^n \cdot \left( {\frac{{b - a}}{2}} \right)^n \cdot \left(
{\bigvee_a^b\left( f \right)} \right)^{1- {\textstyle{1 \over p}}}
\cdot \left\| {u^{(n)}} \right\|_{p,\left[ {a,b} \right]}
\end{multline}
\end{theorem}

\begin{proof}
We repeat the proof of Theorem \ref{thm1}. Now, using the recent
result proved by the first author of this paper; on generalization
of Beesack-Wirtinger inequality \cite{alomari} which reads: If
$h\in \mathfrak{U}^p(I)$   then for all $\xi \in [a,b]$ we have
\begin{multline}
\label{moh}\int_a^b {\left| {h\left( t \right) - h\left( \xi
\right)} \right|^p dt}
\\
 \le\left( {\frac{{p^p \sin ^p \left(
{{\textstyle{\pi  \over p}}} \right)}}{{\pi ^p \left( {p - 1}
\right)}}} \right)^{n}\left[ {\frac{{b - a}}{2} + \left| {\xi  -
\frac{{a + b}}{2}} \right|} \right]^{np} \cdot\int_a^b {\left(
{h^{\left( n \right)} \left( x \right)} \right)^p dx}.
\end{multline}
In case $n=1$, the inequality \eqref{moh} is sharp see
\cite{alomari}.

Therefore, since $u^{(n)}\in L^p[a,b]$ then replacing $h$ by $u$
and the interval $[a,b]$ by the corresponding intervals defines
$u$ in the proof of Theorem \ref{thm1} we get the required result
we shall omit the details.
\end{proof}

The dual assumptions on $f$ and $u$  are considered in the
following two results.
\begin{theorem} \label{thm3}
Let $f\in \mathfrak{U}^p(I)$. Assume that $u:[a,b] \to \mathbb{R}$
has $M$-Lipschitz property on $[a,b]$, then we have the inequality
\begin{multline}
\left| {\mathcal{E}_0  \left( {f,u;x} \right)} \right|
\\
\le 2\lip_{\rm{M}}\left( u \right)\cdot
\left(\frac{b-a}{2}\right)^{1/q}\cdot  \left( {\frac{{p\sin \left(
{{\textstyle{\pi  \over p}}} \right)}}{{\pi \sqrt[p]{{p - 1}}}}}
\right)^n \left[ {\frac{{b - a}}{4} + \left| {x - \frac{{3a +
b}}{4}} \right|} \right]^n \cdot \left\| {f^{(n)}}
\right\|_{p,\left[ {a,b} \right]}\label{eq4.4}
\end{multline}
for all $x \in \left[ {a,\frac{{a + b}}{2}}\right]$. In
particular, we have
\begin{align}
\left| {\mathcal{E}_0 \left( {f,u; \frac{3a+b}{4}} \right)}
\right| \le \lip_{\rm{M}}\left( u \right)\cdot  \left(
{\frac{{p\sin \left( {{\textstyle{\pi \over p}}} \right)}}{{\pi
\sqrt[p]{{p - 1}}}}} \right)^n \frac{{\left( {b - a} \right)^{n +
\frac{1}{q}} }}{{2^{2n + \frac{1}{q} - 1} }} \cdot \left\|
{f^{(n)}} \right\|_{p,\left[ {a,b} \right]}.\label{eq4.5}
\end{align}
\end{theorem}

\begin{proof}
Using the integration by parts formula for
$\mathcal{RS}$-integral, we have
\begin{equation*}
\int_a^{{\textstyle{{a + b} \over 2}}} {\left[ {f\left( x \right)
- f\left( t \right)} \right]du\left( t \right)}  = f\left( x
\right)\left[ {u\left( {\frac{{a + b}}{2}} \right) - u\left( a
\right)} \right] - \int_a^{{\textstyle{{a + b} \over 2}}} {f\left(
t \right)du\left( t \right)},
\end{equation*}
and
\begin{equation*}
\int_{{\textstyle{{a + b} \over 2}}}^b {\left[ {f\left( {a + b -
x} \right) - f\left( t \right)} \right]du\left( t \right)}
=f\left( {a + b - x} \right) \left[ {u\left( b \right) - u\left(
{\frac{{a + b}}{2}} \right)} \right] - \int_{{\textstyle{{a + b}
\over 2}}}^b {f\left( t \right)du\left( t \right)}
\end{equation*}
Adding the above equalities, we have
\begin{multline*}
\int_a^{{\textstyle{{a + b} \over 2}}} {\left[ {f\left( x \right)
- f\left( t \right)} \right]du\left( t \right)}  +
\int_{{\textstyle{{a + b} \over 2}}}^b {\left[ {f\left( {a + b -
x} \right) - f\left( t \right)} \right]du\left( t \right)}
\\
= f\left( x \right) \left[ {u\left( {\frac{{a + b}}{2}} \right) -
u\left( a \right)} \right] + f\left( {a + b - x} \right)\left[
{u\left( b \right) - u\left( {\frac{{a + b}}{2}} \right)} \right]
- \int_a^b {f\left( t \right)du\left( t \right)}.
\end{multline*}
Applying the inequality (\ref{eq2.1}), and then applying the
H\"{o}lder inequality we get
\begin{align*}
&\left| {\mathcal{E}_0  \left( {f,u;x} \right)} \right|
\\
&=\left| {\int_a^{{\textstyle{{a + b} \over 2}}} {\left[ {f\left(
x \right) - f\left( t \right)} \right]du\left( t \right)}  +
\int_{{\textstyle{{a + b} \over 2}}}^b {\left[ {f\left( {a + b -
x} \right) - f\left( t \right)} \right]du\left( t \right)}}
\right|
\\
&\le \left| {\int_a^{{\textstyle{{a + b} \over 2}}} {\left[
{f\left( {x} \right) - f\left( t \right)} \right]du\left( t
\right)} } \right| + \left| {\int_{{\textstyle{{a + b} \over
2}}}^b {\left[ {f\left( {a + b - x} \right) - f\left( t \right)}
\right]du\left( t \right)} } \right|
\\
&\le \lip_{\rm{M}}\left( u \right)\cdot \int_a^{{\textstyle{{a +
b} \over 2}}} {\left| {f\left( {x} \right) - f\left( t \right)}
\right|dt} +\lip_{\rm{M}}\left( u \right)\cdot
\int_{{\textstyle{{a + b} \over 2}}}^b {\left| {f\left( {a + b -
x} \right) - f\left( t \right)} \right|dt}
\\
&\le \lip_{\rm{M}}\left( u \right)\cdot
\left(\frac{b-a}{2}\right)^{1/q}
\left[{\left({\int_a^{{\textstyle{{a + b} \over 2}}} {\left|
{f\left( {x} \right) - f\left( t \right)}
\right|^pdt}}\right)^{1/p} }\right.
\\
&\qquad\qquad \left.{+ \left({\int_{{\textstyle{{a + b} \over
2}}}^b {\left| {f\left( {a + b - x} \right) - f\left( t \right)}
\right|^pdt}}\right)^{1/p} }\right].
\end{align*}
Utilizing \eqref{moh} we can write
\begin{multline*}
\int_a^{{\textstyle{{a + b} \over 2}}} {\left| {f\left( {x}
\right) - f\left( t\right)} \right|^p dt}
\\
\le \left( {\frac{{p^p \sin ^p \left( {{\textstyle{\pi  \over p}}}
\right)}}{{\pi ^p \left( {p - 1} \right)}}} \right)^{n} \cdot
\left( {\frac{{b - a}}{4} + \left| {x - \frac{{3a + b}}{4}}
\right|} \right)^{np}\int_a^{{\textstyle{{a + b} \over 2}}}
{\left| {f^{(n)}\left( {t} \right)} \right|^p dt},
\end{multline*}
and
\begin{multline*}
\int_{{\textstyle{{a + b} \over 2}}}^b {\left| {f\left( {a + b -
x} \right) - f\left( t \right)} \right|^pdt}
\\
\le \left( {\frac{{p^p \sin ^p \left( {{\textstyle{\pi  \over p}}}
\right)}}{{\pi ^p \left( {p - 1} \right)}}} \right)^{n} \cdot
\left( {\frac{{b - a}}{4} + \left| {x - \frac{{3a + b}}{4}}
\right|} \right)^{np}\int_{{\textstyle{{a + b} \over 2}}}^b
{\left| {f^{(n)}\left( {t} \right)} \right|^p dt}.
\end{multline*}
Substituting in these two inequalities in the previous one and
simplify we get the required result and thus the theorem is
proved.
\end{proof}

\begin{theorem}
\label{thm4}Let $f\in \mathfrak{U}^p(I)$. Assume that $u:[a,b] \to
\mathbb{R}$ has $M$-Lipschitz property on $[a,b]$, then we have
the inequality
\begin{multline}
\left| {\mathcal{E}_1  \left( {f,u;x} \right)} \right|
\\
\le 2\lip_{\rm{M}}\left( u \right)\cdot \left( {\frac{{p\sin
\left( {{\textstyle{\pi \over p}}} \right)}}{{\pi \sqrt[p]{{p -
1}}}}} \right)^n \left[ {\frac{{b - a}}{2} + \left| {x - \frac{{a
+ b}}{2}} \right|} \right]^{n+\frac{1}{q}} \cdot \left\| {f^{(n)}}
\right\|_{p,\left[ {a,b} \right]}\label{eq4.6}
\end{multline}
for all $x \in \left[ {a,b}\right]$. In particular, we have
\begin{align}
\left| {\mathcal{E}_1  \left( {f,u; \frac{a+b}{2}} \right)}
\right|  \le \lip_{\rm{M}}\left( u \right)\cdot \left(
{\frac{{p\sin \left( {{\textstyle{\pi \over p}}} \right)}}{{\pi
\sqrt[p]{{p - 1}}}}} \right)^n \frac{{\left( {b - a} \right)^{n +
\frac{1}{q}} }}{{2^{n + \frac{1}{q} - 1} }} \cdot \left\|
{f^{(n)}} \right\|_{p,\left[ {a,b} \right]}.\label{eq4.7}
\end{align}
\end{theorem}

\begin{proof}
Using the integration by parts formula for
$\mathcal{RS}$-integral, we have
\begin{equation*}
\int_a^{x} {\left[ {f\left( a \right) - f\left( t \right)}
\right]du\left( t \right)}  = f\left( a \right)\left[ {u\left( {x}
\right) - u\left( a \right)} \right] - \int_a^{x} {f\left( t
\right)du\left( t \right)},
\end{equation*}
and
\begin{equation*}
\int_{x}^b {\left[ {f\left( {b} \right) - f\left( t \right)}
\right]du\left( t \right)} =f\left( {b} \right) \left[ {u\left( b
\right) - u\left( {x} \right)} \right] - \int_{x}^b {f\left( t
\right)du\left( t \right)}.
\end{equation*}
Adding the above equalities, we have
\begin{multline*}
\int_a^{x} {\left[ {f\left(a \right) - f\left( t \right)}
\right]du\left( t \right)}  + \int_{x}^b {\left[ {f\left( {b}
\right) - f\left( t \right)} \right]du\left( t \right)}
\\
= f\left( a \right) \left[ {u\left( {x} \right) - u\left( a
\right)} \right] + f\left( {b} \right)\left[ {u\left( b \right) -
u\left( {x} \right)} \right] - \int_a^b {f\left( t \right)du\left(
t \right)}.
\end{multline*}
Following the same steps in the proof of Theorem \ref{thm3} we get
the required result.
\end{proof}

\begin{remark}
The general error term $\mathcal{E}_{\alpha} \left( {f,u;x}
\right)$ has the form
\begin{align*}
\mathcal{E}_{\alpha} \left( {f,u;x}
\right)=\left(1-\alpha\right)\mathcal{E}_{0} \left( {f,u;x}
\right)+\alpha \mathcal{E}_{1} \left( {f,u;x} \right).
\end{align*}
for all $\alpha \in \left[0,1\right]$ and $x \in
\left[{a,\frac{a+b}{2}}\right]$.

In particular, the error of Simpson like quadrature formula is
obtained from the identity
\begin{align*}
\mathcal{E}_{\frac{1}{3}} \left( {f,u;\frac{a+b}{2}}
\right)=\frac{2}{3}\mathcal{E}_{0} \left( {f,u;\frac{a+b}{2}}
\right)+\frac{1}{3}\mathcal{E}_{1} \left( {f,u;\frac{a+b}{2}}
\right).
\end{align*}
 Thus, by \eqref{eq4.5} and \eqref{eq4.7} we get
\begin{align*}
\left| {\mathcal{E}_{\frac{1}{3}}  \left( {f,u;\frac{a+b}{2}}
\right)}\right| &\le   \lip_{\rm{M}}\left( u \right) \frac{{\left(
{b - a} \right)^{n + \frac{1}{q}} }}{{2^{n + \frac{1}{q} - 1}
}}\left( {\frac{{p\sin \left( {{\textstyle{\pi \over p}}}
\right)}}{{\pi \sqrt[p]{{p - 1}}}}} \right)^n    \cdot \left\|
{f^{(n)}} \right\|_{p,\left[ {a,b} \right]}.
\end{align*}
\end{remark}

\begin{remark}
In all above results the best error estimates hold with $L^2$-norm
i.e., $p=q=2$.
\end{remark}

\begin{remark}
To get $L^p$-bounds with bounded variation integrators one may
apply Lemma \ref{lemma2} instead of \eqref{eq2.1} in the proofs of
Theorems \ref{thm3} and \ref{thm4}. Also, we may apply Lemma
\ref{lemma1} instead of H\"{o}lder inequality in the proofs of
Theorems \ref{thm3} and \ref{thm4}.
\end{remark}

\begin{remark}
One may apply the unused results in Section \ref{sec2} to obtain
more error bounds.
\end{remark}

\begin{remark}\label{rem12}
In the presented quadrature,  high degree of accuracy occurred
significantly  with less error estimations when higher derivatives
are assumed  on very small scale of intervals. Particularly, if
one assumes that $f^{(m)} \in L_2[a,b]$ and $b-a \le \frac{1}{2^m
m!}$ $\left( {m\in \mathbb{N}} \right)$ then as $m$ increases all
obtained error estimations become very small. Hence, the presented
results  are recommended to be applied for small scale of
intervals or to be applied as composite rules.
\end{remark}

Let $f:[a,b]\rightarrow \mathbb{R}$, be a twice
differentiable mapping such that $f^{\prime \prime }\left( x\right) $ exists and bounded on $%
(a,b)$. Then  the trapezoidal rule reads
\begin{align}
\label{trap}  \int_a^b {f\left( x \right)dx}  = \left( {b - a}
\right)\frac{{f\left( a \right) + f\left( b \right)}}{2} -
\frac{{\left( {b - a} \right)^3 }}{{12}}
f^{\prime\prime}\left({\xi}\right), \,\, {\text{ for
\,\,some}}\,\,a<\xi<b
\end{align}

To improve our Remark \ref{rem12},  we give a numerical example by
comparing our formula \eqref{eq4.5} with Trapezoidal rule
\eqref{trap}.  It is unusual to compare two approximations
evaluated by two different norms unless we get a very close
estimations or we don't have a well-know rule  to compare with.

Let $\left[{a,b}\right]=\left[{0,\frac{1}{2^n n!}}\right]$ with
$u(t)=t$. In viewing \eqref{eq4.5} we get
\begin{align}
\left| {\mathcal{E}_0 \left( {f,t; \frac{1}{2^{n+2} n!}} \right)}
\right| \le   \frac{2^{1/2}}{(2\pi)^n}  \cdot \left( {\frac{1}{2^n
n!}} \right)^{n + \frac{1}{2}}\cdot \left\| {f^{(n)}}
\right\|_{2}.\label{eq4.9}
\end{align}
Employing \eqref{eq3.3} for the particular choice $n=2$, we get
\begin{align}
\int_0^{\frac{1}{8}} {f\left( t \right)dt} =\Phi _0\left(
{f,t;\frac{1}{ 32}} \right) -\mathcal{E}_0  \left( {f,t;\frac{1}{
32}} \right). \label{eq4.10}
\end{align}
Consider $f(t)=\exp(-t^2)$, $t\in \left[{0,\frac{1}{8}}\right]$.
Then, we have the exact value
\begin{align}
\int_0^{\frac{1}{8}} {f\left( t
\right)dt}=\frac{\sqrt\pi}{2}\erf\left(\frac{1}{8}\right)
=0.1243519988.
\end{align}
 Employing \eqref{eq4.10} we get $\Phi _0\left(
{f,t;\frac{1}{ 32}} \right)= 0.1243920852$ and $\mathcal{E}_0
\left( {f,t;\frac{1}{ 32}} \right)=1.482678376 \times 10^{-15}$ so
that $\int_0^{\frac{1}{8}} {f\left( t \right)dt}=0.1243920852$.
However, applying the Trapezoidal rule \eqref{trap} we get
$\int_0^{\frac{1}{8}} {f\left( t \right)dt}=0.1243487939$. By
comparing the  two evaluations, the absolute error in
\eqref{eq4.10} is $4.00864\times 10^{-5}$ and in \eqref{trap} is
$3.2049\times 10^{-6}$. Taking into account that we compare two
approximations via two different norms.

\centerline{}

\centerline{}

\end{document}